\documentclass[11pt]{article} 

\usepackage{amsmath}
\usepackage{amsthm}
\usepackage{amsfonts}
\usepackage{mathrsfs}
\usepackage{stmaryrd}
\usepackage{setspace}
\usepackage{fullpage}
\usepackage{amssymb}
\usepackage{breqn}
\usepackage{enumitem}
\usepackage{bbold} 
\usepackage{authblk}
\usepackage{comment}
\usepackage{hyperref}
\usepackage{pgf,tikz}
\usepackage{graphicx}
\usepackage{colonequals}
\usepackage{microtype}
\usepackage{cite}
\usetikzlibrary{decorations.pathreplacing,arrows}
\tikzstyle{vertex}=[circle,draw=black,fill=black,inner sep=0,minimum size=3pt,text=white,font=\footnotesize]

\newtheorem{definition}{Definition}[section]

\newtheorem{theorem}[definition]{Theorem}

\newtheorem{lemma}[definition]{Lemma}

\numberwithin{equation}{section}

\newcommand{\ignore}[1]{}

\providecommand{\keywords}[1]
{
  \small	
  \textbf{{Keywords}} #1
}

\if11     
\usepackage[mathlines]{lineno}
\newcommand*\patchAmsMathEnvironmentForLineno[1]{%
  \expandafter\let\csname old#1\expandafter\endcsname\csname #1\endcsname
  \expandafter\let\csname oldend#1\expandafter\endcsname\csname end#1\endcsname
  \renewenvironment{#1}%
     {\linenomath\csname old#1\endcsname}%
     {\csname oldend#1\endcsname\endlinenomath}}%
\newcommand*\patchBothAmsMathEnvironmentsForLineno[1]{%
  \patchAmsMathEnvironmentForLineno{#1}%
  \patchAmsMathEnvironmentForLineno{#1*}}%
\AtBeginDocument{%
\patchBothAmsMathEnvironmentsForLineno{equation}%
\patchBothAmsMathEnvironmentsForLineno{align}%
\patchBothAmsMathEnvironmentsForLineno{flalign}%
\patchBothAmsMathEnvironmentsForLineno{alignat}%
\patchBothAmsMathEnvironmentsForLineno{gather}%
\patchBothAmsMathEnvironmentsForLineno{multline}%
}
\fi

\def\inst#1{$^{#1}$}

\begin{document}

\title{Linear Secret-Sharing Schemes for $k$-uniform access structures }

\author{Younjin Kim\inst{1}\thanks{Y.K. was supported by Basic Science Research Program through the National Research Foundation of Korea(NRF) funded by the Ministry of Education (Grant No. 2017R1A6A3A04005963).}, Jihye Kwon\inst{2}, and Hyang-Sook Lee\inst{3}\thanks{H.L. was supported by Basic Science Research Program through the National Research Foundation of Korea(NRF) funded by the Ministry of  Education (Grant No. 2019R1A6A1A11051177).}}

\maketitle

\begin{center}
{\footnotesize
\inst{1} 
Institute of Mathematical Sciences, Ewha Womans University, Seoul, South Korea \\
\texttt{younjinkim@ewha.ac.kr}
\\\ \\
\inst{2} 
Department of Mathematics, Ewha Womans University, Seoul, South Korea \\
\texttt{jhkwon74@ewhain.net}
\\\ \\

\inst{3} 
Department of Mathematics, Ewha Womans University, Seoul, South Korea \\
\texttt{hsl@ewha.ac.kr}
\\\ \\

}
\end{center}

\begin{abstract}
A {\it $k$-uniform hypergraph} $\mathcal{H}=(V, E)$ consists of a
set $V$ of vertices and a set $E$ of hyperedges ($k$-hyperedges), which is a family of $k$-subsets of $V$.  
A {\it forbidden $k$-homogeneous  (or forbidden $k$-hypergraph)} access structure $\mathcal{A}$ is represented by  a $k$-uniform hypergraph $\mathcal{H}=(V,  E)$ and has the following property: a set of vertices (participants) can reconstruct the secret value from their shares in the secret sharing scheme if they are connected by a $k$-hyperedge or their size is at least $k+1$.  A forbidden $k$-homogeneous access structure has been studied by many authors ~\cite{AA, ABFNP, BKN,SS} under the terminology of $k$-uniform access structures. In this paper, we provide efficient constructions 
 on the total share size of linear secret sharing schemes for sparse and dense  $k$-uniform access structures   for a constant $k$ using the hypergraph decomposition technique and  the monotone span programs.\\
 
 \keywords{secret sharing schemes, graph access structures, hypergraph decomposition}
 \end{abstract}

\section{Introduction}

A secret sharing scheme is a tool used in many cryptographic protocols. A secret sharing scheme involves a dealer who has a secret, a set of $n$ participants, and a collection $\mathcal{F}$ of subsets of participants defined as the access structure. A secret sharing scheme for $\mathcal{F}$ is a method by which the dealer distributes shares of a secret value $k$ to the set of $n$ participants such that any subset in $\mathcal{F}$ can reconstruct the secret value $k$ from their shares and any subset not in $\mathcal{F}$ cannot reveal any information about the secret value $k$. When any subset in $\mathcal{F}$ can reconstruct the secret value $k$ from their shares by using a linear mapping, the secret sharing scheme is called a {\it linear secret sharing scheme}. The {\it qualified subsets} in the secret sharing scheme is defined as the subsets of participants who can reconstruct the secret value $k$ from their shares. A collection of  qualified subsets of participants called the {\it access structure} of the secret sharing scheme. In  other words, the {\it unqualified subsets}  or {\it forbidden subsets} in the secret sharing scheme are defined as the subsets of participants who cannot have any information about  the secret value $k$ from their shares. \\

In 1979, Shamir~\cite{AS} introduced a $(t,n)$-{\it threshold secret sharing scheme} as the first works about the secret sharing, in which the qualified subsets  are formed by all the subsets with at least $t$ participants in a set of $n$ participants and the size of each share is  the size of the secret. It means that $(t,n)$-{\it threshold secret sharing scheme} is determined by the basis consisting of all subsets with exactly $t$ different participants from a set of $n$ participants. There have been further constructions of secret sharing schemes for any access structures, and in 1987, Ito, Saito, and Nishizeki~\cite{ISN} constructed secret sharing schemes for general access structures. However, their constructions are very inefficient because the size of the shares much larger than the size of the secret in general. Later, in 1988, Benaloh and Leichter~\cite{BL} constructed a much more efficient  secret sharing scheme for general access structure based on monotone formulate than the scheme of Ito, Saito, and Nishizeki~\cite{ISN}.  \\

All the above secret-sharing schemes are linear in which the secret is an element of the field and each share is a vector over the field whose each coordinate is expressed as a linear  combination of the secret, and the coordinates of the random strings which are taken from  some finite field. In 1993, Karchmer and Wigderson~\cite{KW} introduced the monotone span programs from which the linear secret sharing schemes can be constructed.  They obtained that every monotone span program over finite fields implies a linear secret sharing scheme for an access structure consisting of all sets accepted by the monotone span program.
Later, in 1993, Bertilsson and Ingemarsson~\cite{BI} generalized their linear schemes derived from the monotone span programs to the multilinear schemes  based on the generalized monotone span programs in which  the secret is some vector over the  field.  The best-known lower bound on the total share size of secret sharing schemes realizing a general access structure was given by Csirmaz~\cite{LC1} in 1997. Also,  the best-known upper bound on the total share size of secret sharing scheme  realizing a general access structure was given by  Applebaum, Beimel, Farr\'as, Nir, and Peter~\cite{ABFNP} in 2020, which
 is highly inefficient with the  size $2^{0.637n}$. \\
 

An access structure is defined as a {\it graph access structure} determined by a graph $G=(V, E)$ if  a pair of  vertices connected by an edge can reconstruct the secret and the set of non-adjacent vertices in the graph $G$ does not get any information on the secret. The motivation for studying graph secret sharing schemes is that they are simpler than secret sharing schemes for general access structures and  later generalized to general access structures. Secret sharing schemes realizing graph access structures were studied in many papers~\cite{BFM,BSSV,BD,BS,CSGV,LC2,LC3,LC4,CL,CT,MD,FKMP} Also, many authors were interested in forbidden graph access structures as specific families of access structures.
 An access structure is defined as a {\it forbidden graph access structure} determined by a graph $G=(V, E)$ if  a pair of  vertices can reconstruct the secret if it is connected by an edge or its size is at least $3$. In 2014, Beimel, Ishai, Kumaresan, and Kushilevitz~\cite{BIKK} constructed a secret sharing scheme realizing all forbidden graph access structures with the total share size $O(n^{3/2})$. Later, in
 2015, a linear secret sharing scheme 
for all forbidden graph access structures was given by
Gay, Kerenidis, and Wee~\cite{GKW}  in which  
 the total share size is $O(n^{3/2})$.  Recently, in 2017, Liu, Vaikuntanathan, and Wee~\cite{LVW} proved that every forbidden graph access structure could be realized by a non-linear secret sharing scheme with the total share size $n^{1+o(1)}$. For the forbidden dense graph access structures  having at least ${{n} \choose {2}}-n^{1+\beta}$ edges, where $0 \leq \beta < {\frac{1}{2}}$, Beimel, Farras, and Peter~\cite{BFP} constructed a linear secret sharing scheme  with the total share size $O(n^{7/6+2\beta/3})$. Later, in 2020, Beimel, Farras, Mintz, and Peter~\cite{BFMP} provided efficient constructions on the share size of linear secret schemes for forbidden sparse and dense  graph access structures based on the monotone span programs.\\


A hypergraph is a generalization of a graph in which hyperedges may connect  more than two vertices. A  $k$-uniform hypergraph is a hypergraph in which each hyperedge has exactly $k$ vertices.
 An access structure is defined as a {\it $k$-hypergraph access structure} determined by a $k$-uniform hypergraph $\mathcal{H}$ if  the set of  vertices connected by a $k$-hyperedge can reconstruct  the secret and the set of non-adjacent vertices in the hypergraph $H$ does not get any information on the secret. The access structures of these schemes are also called {\it $k$-homogeneous}.  For example, graph access structures are $2$-homogeneous access structures. A $k$-homogeneous access structure is determined by the family of minimal qualified subsets with exactly $k$ different participants or $k$-uniform hypergraphs $\mathcal{H}(V,E)$, where $V$ is a vertex set, and $E \subseteq 2^V$ is an edge set of hyperedges of cardinality $k$.  
The secret sharing schemes for $k$-homogeneous access structures have been constructed by many authors based on various techniques.  In 1990, Benaloh and Leichter~\cite{BL} constructed a secret sharing scheme for the $k$-homogeneous access structure with total share size $O(n^k/ \log n)$. 
 For the dense $k$-homogeneous access structure,  a much more efficient  secret sharing scheme was constructed by 
Beimel, Farras, and Mintz~\cite{BFM}  in 2012, in which the total share size is  $\tilde{O}(2^kk^kn^{2+\beta})$.
 Recently, in 2020, Beimel and Farras~\cite{BF1} constructed a secret sharing scheme for almost all $k$-homogeneous access structures with maximum share size $2^{\tilde{O}(\sqrt{k\log n})}$. \\

An access structure is defined as {\it forbidden $k$-homogeneous  (or forbidden $k$-hypergraph)} determined by a hypergraph $\mathcal{H}$ if  a set of vertices can reconstruct the secret if it is connected by a $k$-hyperedge or its size is at least $k+1$. We study the complexity of realizing a forbidden $k$-hypergraph access structure by linear secret sharing schemes. A forbidden $k$-homogeneous access structure has been studied by many authors~\cite{AA, ABFNP, BKN,SS} under the terminology of $k$-uniform access structures. Recently, in 2018, Applebaum and Arkis constructed an efficient secret sharing scheme for  $k$-uniform access structures  using  multiparty Conditional Disclosure of Secrets (CDS). Later, in 2018, Beimel and Peter~\cite{BP} obtained that every $k$-uniform access structure with a binary secret could be realized by a secret sharing scheme in which the share size $ \min  \{ (O(n/k))^{(k-1)/2}, O(n\cdot 2^{n/2})\} $. By improving their result, in 2019, Applebaum, Beimel, Farras, Nir, and Peter~\cite{ABFNP} obtained that every $k$-uniform access structure with a binary secret could be realized by a secret sharing scheme in which the share size  is $ 2^{\tilde{O}(\sqrt{k\log n})}$ by combining the CDS protocol and transformations. In 2020, Beimel, Farras, Mintz, and Peter~\cite{BFMP} obtained the lower bound on the max share size for sharing a one-bit secret in every linear secret sharing scheme realizing $k$-uniform access structures using CDS protocol.
In this paper, we provide efficient constructions 
 on the share size of linear secret sharing schemes for sparse and dense  $k$-uniform access structures (or forbidden $k$-homogeneous access structures)  for a constant $k$ using the hypergraph decomposition technique and  the monotone span programs
 as follows.
 
 \newpage


\begin{theorem}\label{main:mainthm1}  Let $\Gamma$ be a sparse  $k$-uniform access structure  whose size is at  most $n^{1+\beta}$, where $0 \leq \beta <1$.
Then there exists a linear secret sharing scheme for  an access structure $\Gamma$  with the total share size 

$$O\left(n^{\frac{k^2-3k+2}{k^2-2k+2}+ \left(\frac{k^2-3k+3}{k^2-2k+2}\right)\beta} \log^{k+1} n \right).$$ \\
\end{theorem}

\begin{theorem}\label{main:mainthm2}
Let $\Gamma$ be a dense  $k$-uniform access structure whose size is at least $ {{n}\choose{k}}- n^{1+\beta}$, where $0 \leq \beta <1$.
Then there exists a linear secret sharing scheme for  an access structure $\Gamma$  with the total share size 

$$O\left(n^{\frac{k^2-3k+2}{k^2-2k+2}+ \left(\frac{k^2-3k+3}{k^2-2k+2}\right)\beta} \log^{k+1} n \right).$$ \\
\end{theorem}

Our paper is organized as follows. In Section $2$, 
 we introduce the definition of Secret Sharing Scheme and  two interesting secret sharing schemes which are Shamir's Threshold Secret Sharing Scheme and  Monotone Span Programs. In Section $3$, we introduce several access structures related to this paper.
 In Section $4$ and Section $5$, we present the results and lemmas, which are necessary for proving our main theorems.
In Section $6$ and Section $7$, we give the proof of  Theorem~\ref{main:mainthm1} and Theorem~\ref{main:mainthm2}.

\section{Secret Sharing Scheme}


 \noindent A secret sharing scheme involves a dealer who has a secret, a set of $n$ participants, and a collection $\mathcal{F}$ of subsets of participants defined as the access structure. A secret sharing scheme for $\mathcal{F}$ is a method by which the dealer distributes shares of a secret value $k$ to the set of $n$ participants such that any qualified subset in $\mathcal{F}$ can reconstruct the secret value $k$ from their shares and any unqualified subset not in $\mathcal{F}$ cannot reveal any information about the secret value $k$. By using the entropy function we define the secret sharing scheme. \\
 
\noindent For the given random variable $X$, we define the entropy of $X$ as 
$$H(X) = - \sum Pr(X=x) \log Pr(X=x),$$
\noindent where the sum is taken over all values $x$ and $Pr(X=x) > 0$. For the two random variables $X$ and $Y$, we define the conditional entropy as 
$H(X|Y) = H(XY)-H(Y)$. Clearly  we obtain that $ 0 \leq H(X|Y) \leq H(X)$ and the following two properties hold:
(1) two random variables $X$ and $Y$ are independent  iff $H(X|Y) = H(X)$ and (2) the value of $Y$ implies the value of $X$ iff $H(X|Y)=0$. \\

Let $P$ be a set of $n$ participants where $P = \{p_1, p_2, \cdots, p_n\}$. 
Let $\mathcal{F}$ be a collection of subsets of participants defined as an access structure.
It means that sets in $\mathcal{F}$ are qualified and sets not in $\mathcal{F}$ are unqualified.
Assume that there is the probability distribution on the domain of secrets. We also consider the probability distribution on the vector of share of any subset of $n$ participants. We define the random variable denoting the secret as $S$. Let us define the random variable denoting the share values of any subset $X$ of $n$ participants as $S_X$. 

\begin{definition}[Secret Sharing Scheme] For the given probability distribution on the secrets,  we say that a distribution scheme is a {\it Secret Sharing Scheme}  realizing an access structure  if the following two requirements hold:\\

\noindent {\bf Correcteness.} For every qualified set $B \in \mathcal{F}$,
$$ H(S \ | \ S_B) = 0.$$

\noindent {\bf Privacy.} For every unqualified set $T  \not \in \mathcal{F}$,
$$ H(S \ | \ S_T) = H(S).$$ 

\end{definition}

Th one parameter for measuring the efficiency of a secret sharing scheme is the information rate, which is defined as 
the ratio between the length of secret and the maximum length of the shares given to the participants. Since the length of any share is greater than or equal to the length in a secret sharing scheme, the information rate can not be greater than one. Secret sharing schemes with an information rate equal to one are called {\it ideal secret sharing schemes}. The following Shamir's threshold secret sharing scheme is ideal.

\subsection{Shamir's Threshold Secret Sharing Scheme\\  }

In 1979, Shamir~\cite{AS} introduced a $(t,n)$-{\it threshold secret sharing scheme} as the first works about the secret sharing, in which the qualified subsets  are formed by all the subsets with at least $t$ participants in a set of $n$ participants. 
Let $P$ be a set of $n$ participants where $P = \{p_1, p_2, \cdots, p_n\}$. 
Let $\mathcal{F}_t = \{ A \subseteq \{p_1, p_2, \cdots, p_n \} \ | \ |A| \geq t\}$, where $1 \leq t \leq n $ is an integer,
be a collection of subsets of participants defined as an access structure.
It means that sets in $\mathcal{F}_t$ are qualified and sets not in $\mathcal{F}_t$ are unqualified. The domain of secrets and shares is defined as the elements of a finite field $\mathbf{F}_q$ for some prime power $q > n$.\\

\noindent Let $\alpha_1, \alpha_2, \cdots, \alpha_n$ be $n$ elements of a finite field $\mathbf{F}_q$ corresponding to each participant $p_i$, where 
$ 1 \leq i \leq n$. To distribute shares of a secret value $k$ to the set of $n$ participants, a dealer chooses $t-1$ random elements $a_1, a_2, \cdots a_{t-1}$ uniformly and independently from the finite field $\mathbf{F}_q$ to define the following polynomial of degree at most $t-1$.

$$ P(x)= k + \sum_{i=1}^{t-1} a_i x^i$$

\noindent where $k$ is a secret which is in the field $\mathbf{F}_q$.\\

\noindent  We define the value of $s_j = P(\alpha_j)$ as the share of each participant $p_j$, where $1\leq j \leq t$.
Note that the size of each share is  same as the size of the secret.  We claim that every set $B= \{ p_{i_1}, p_{i_2}, \cdots, p_{i_t} \}$ of size $t$ among $n$ participants
can reconstruct the secret value $k$. Let us consider the following polynomial of degree at most $t-1$

$$ Q(x) = \sum_{l=1}^t s_{i_l} \prod_{1\leq j \leq t, \\ j \neq l } \frac {\alpha_{i_j} - x}{\alpha_{i_j}-\alpha_{i_l}}.$$ 

\noindent Note that $P(\alpha_{i_l}) = s_{i_l} = Q (\alpha_{i_l})$ for all $1 \leq l \leq t$. By using the Lagrange's interpolation theorem, the polynomial  $Q$ are equivalent to the polynomial  $P$ and $Q(0) = P(0) =k$. It means that every set $B$ can reconstruct the secret value $k$ by computing

$$ k = Q(0) = \sum_{l=1}^t s_{i_l} \prod_{1\leq j \leq t, \\ j \neq l } \frac {\alpha_{i_j}}{\alpha_{i_j}-\alpha_{i_l}}$$ \\

\noindent which is  a linear combination of the shares and  $\prod_{1\leq j \leq t, \\ j \neq l } \frac {\alpha_{i_j}}{\alpha_{i_j}-\alpha_{i_l}}$ depends only on the set $B$. Therefore, it satisfies the first requirement.\\

\noindent For any unqualified set $T = \{ p_{j_1}, p_{j_2} , \cdots, p_{j_{t-1}} \}$ with size $t-1$, 
 there exists an unique polynomial $P_a$ of degree at most $t-1$ with $P_a(0)=a$ and $P_a(\alpha_{j_l}) = s_{j_l}$ for every secret $a \in \mathbf{F}_q$, where $1 \leq j \leq t-1$, such that   the probability 
 computing a vector of shares is same as $\frac{1}{q^{t-1}}$
 for every secret $a \in \mathbf{F}_q$. Therefore, it satisfies the second requirement.

\subsection{The Monotone Span Programs    \\ }

  In 1993, Karchmer and Wigderson~\cite{KW} introduced the monotone span programs from which the linear secret sharing schemes can be constructed.  In the linear secret sharing schemes, the secret is an element of the field and each share is a vector over the field whose each coordinate is expressed as a linear  combination of the secret and the coordinates of the random strings which are taken from  some finite field.
Let $P$ be a set of $n$ participants where $P = \{p_1, p_2, \cdots, p_n\}$. 
Let $\mathcal{F}$ be a collection of subsets of participants defined as an access structure.

\begin{definition}[Monotone Span Program] Let $\mathbf{F}$ be a finite field. Let $M$ be an $\alpha \times \beta$ matrix over the field $\mathbf{F}$ where $\rho : \{1,2, \cdots,  \alpha \} \rightarrow \{ p_1, p_2, \cdots, p_n\} $ labels each row of the Matrix $M$ by one participant among $n$ participants. For any subset $X \subseteq \{p_1, p_2, \cdots, p_n \}$, we denote $M_X$ be the sub matrix obtained from the matrix $M$ by restricting all $\alpha $ rows to the rows labeled by participants in $X$.
A {\it Monotone Span Program} $ \mathcal{M}$ over the finite field $\mathbf{F}$ consists of above triple $(\mathbf{F}, M, \rho)$.  We say that the Monotone Span Program $\mathcal{M}$ accepts the set $B \subseteq \{p_1, p_2, \cdots, p_n \}$  if the rows of the submatrix $M_B$  span 
a vector ${\bf{e}}_1 =(1,0,0, \cdots, 0)$ or a target vector $v$.
\end{definition}

 In the following Lemma,  Karchmer and Wigderson~\cite{KW}  proved that every monotone span program over finite fields implies a linear secret sharing scheme for an access structure consisting of all sets which are accepted by the monotone span program $\mathcal{M}$.

\begin{lemma}[Karchmer and Wigderson~\cite{KW}]

Let $\mathcal{M}$ be a monotone span program accepting all sets in an access structure $\mathcal{F}$ consisting of a finite field $\mathbf{F}$, an $\alpha \times \beta$ matrix $M$, and a function $\rho$ labeling $j$-th rows of $M$  by a participant $p_j$.
Then there exists  a linear secret sharing scheme for an access structure $\mathcal{F}$ such that 
 the share of a participant $p_j$ is a vector in $\mathbf{F}^{\alpha_j}$.
\end{lemma}

\begin{proof} For the given monotone span program $\mathcal{M}$, we define a linear secret sharing scheme as follows. 
\noindent Let $\alpha_1, \alpha_2, \cdots, \alpha_n$ be $n$ elements of a finite field $\mathbf{F}$ corresponding to each participant $p_i$, where 
$ 1 \leq i \leq n$. To distribute shares of a secret value $k$ to the set of $n$ participants, a dealer chooses $b-1$ random elements $r_2, r_3, \cdots, r_{b}$ uniformly and independently from the finite field $\mathbf{F}$ to define $r = (k, r_2, r_3, \cdots, r_b)$.
We define the value of $s_j$ satisfying the following equation as the share of each participant $p_j$, where $1\leq j \leq t$.
$$ (s_1, s_2, \cdots, s_{\alpha}) = Mr.$$

  We claim that every set $B$ in an access structure $\mathcal{F}$
can reconstruct the secret value $k$. Let us consider  the submatrix $M_B$ obtained by restricting all $\alpha $ rows  of the matrix $M$ to the rows labeled by participants in $B$. Since there exists some vector $v$ such that $e_1 = vM_B$, the rows of the submatrix $M_B$  span 
the vector ${\bf{e}}_1 =(1,0,0, \cdots, 0)$. Therefore, 
we say that the Monotone Span Program $\mathcal{M}$ accepts every set $B$ in an access structure $\mathcal{F}$. The shares $s_i$ of every participant $p_i$ in the set $B$ satisfies 
$ (s_1, s_2, \cdots, s_{|B|}) = M_Br.$ Then we conclude that 

$$ v(M_Br) = (v M_B)r= e_1\cdot r=k.$$

\noindent It means that every set $B$ in an access structure $\mathcal{F}$
can reconstruct the secret value $k$ by computing $v(M_Br)$. \\ 

Now we claim that for every set $T \not \in \mathcal{F}$ the rows of $M_T$ do not span the vector $e_1$. We denote  the matrix containing all rows of $M_T$ and additional row $e_1$ by $\left( \begin{array}{c} M_T \\ e_1 \end{array} \right)$. From the property $ {\text{rank}} \left(M_T\right) < {\text{rank}} \left( \begin{array}{c} M_T \\ e_1 \end{array} \right) $,  we say that there exists some vector $w \in \mathbf{F}^b$ such that
$ (M_T)w = 0 $ and $e_1\cdot w=1$. \\

To distribute shares of a secret value $k\in \mathbf{F}$ to every participant, we define $r = (k, r_2, r_3, \cdots, r_b)$ with $b-1$ random elements which are chosen by a dealer. Since the shares of every participant $p_i$ in the set $T$ are $(s_1, s_2, \cdots, s_{|T|}) = M_T r$, we conclude that

$$  (M_T)r' = M_T (r - kw) = (M_T)r + k(M_T)w = (M_T)r = \left(s_1, s_2, \cdots, s_{|T|} \right)$$

\noindent for $r' = r - kw $.\\

\noindent It means that the probability that the shares are generated is the same for every secret $k \in \mathbf{F}$. 
\end{proof}

\section{Access Structures}

The qualified subsets in the secret sharing scheme are the subsets of participants who can
reconstruct the secret value from their shares. A collection of qualified subsets of participants
 called the access structure of the secret sharing scheme. In other words, the unqualified subsets in the secret sharing scheme are the collection of participants who cannot
 have any information about the secret value  from their shares. In any secret sharing scheme,
 an access structure is considered to be monotone, in which the superset of the qualified subsets is
 a qualified subset and determined by the family of minimal qualified subsets of participants. A
 collection of minimal qualified subsets of participants called the basis of the access structure. In a secret sharing scheme, every participant must be in at least one minimal qualified subset amongst them.\\

\subsection{Graph Access Structures} A graph access structure is represented by a graph $G=(V, E)$  and has the following property:
a pair of  vertices (participants) connected by an edge can reconstruct the secret  value from their shares in the secret sharing scheme and  independent vertices (participants) in the graph $G$ does not get any information on the secret value.  A trivial secret sharing scheme realizing a graph access structure shares the secret value independently for each edge in which the total share size  is $O(n^2)$. An improved  linear secret sharing scheme realizing every graph access structure was given by Erdos and Pyber~\cite{EP} in which the total share size is $O(n^2/\log n)$. Secret sharing schemes realizing graph access structures have been studied by many authors~\cite{BFM,BSSV,BD,BS,CSGV,LC2,LC3,LC4,CL,CT,MD,FKMP}. \\

The motivation for studying  secret sharing schemes for graph access structures is that they are simpler than secret sharing schemes for general access structures and  later generalized to general access structures. Recently, in 2020, Peter~\cite{NP} proved that for every constant $ 0 < c < \frac{1}{2}$ a secret sharing scheme  for graph access structure with the share size $O(n^c)$  implies a secret sharing scheme for any access structure with the share size $2^{O(0.5+c)n}$. It means that improved secret sharing schemes in the share size for all graph access structures will result in improved secret sharing schemes for all access structures.  For the dense graph access structures  having at least ${{n} \choose {2}}-n^{1+\beta}$ edges, where $0 \leq \beta < 1$, Beimel, Farras, and Mintz~\cite{BFM} constructed a linear secret sharing scheme  with the total share size $\tilde{O}(n^{5/4+3\beta/4})$.

\subsection{Forbidden Graph Access Structures}

In 1997, Sun and Shieh~\cite{SS} introduced secret sharing schemes for forbidden graph access structures in which the participants correspond to the vertices of the graph $G$ and  a pair of  vertices can reconstruct the secret value from their shares in the secret sharing scheme if they are connected by an edge or their size is at least $3$. Secret sharing schemes for graph access structures  and  forbidden graph access structures are very similar. However, graph access structures are harder to realize than forbidden graph access structures. From the given secret sharing scheme realizing graph access structures, we can construct secret sharing schemes realizing forbidden graph access structures by giving a share of the graph secret sharing schemes and a share of $(3,n)$-threshold secret sharing schemes. However, the total share size of the new scheme is slightly greater than the former scheme. It means that  bounds  on the share size of secret sharing schemes for graph access structures imply the bounds on the share size of secret sharing schemes for forbidden graph access structures. \\

 In 2014, Beimel, Ishai, Kumaresan, and Kushilevitz~\cite{BIKK} constructed a secret sharing scheme realizing all forbidden graph access structures with the total share size $O(n^{3/2})$. Later, in
 2015, a linear secret sharing scheme 
for all forbidden graph access structures was given by
Gay, Kerenidis, and Wee~\cite{GKW}  in which  
 the total share size is $O(n^{3/2})$.  Recently, in 2017, Liu, Vaikuntanathan, and Wee~\cite{LVW} proved that every forbidden graph access structure could be realized by a non-linear secret sharing scheme with the total share size $n^{1+o(1)}$. For the forbidden dense graph access structures  having at least ${{n} \choose {2}}-n^{1+\beta}$ edges, where $0 \leq \beta < {\frac{1}{2}}$, Beimel, Farras, and Peter~\cite{BFP} constructed a linear secret sharing scheme  with the total share size $O(n^{7/6+2\beta/3})$. Later, in 2020, Beimel, Farras, Mintz, and Peter~\cite{BFMP} provided efficient constructions on the share size of linear secret sharing schemes for forbidden sparse and dense  graph access structures based on the monotone span programs.

\subsection{$k$-homogeneous access structures}
A hypergraph is a generalization of a graph in which hyperedges may connect  more than two vertices. A  $k$-uniform hypergraph (or $k$-hypergraph) is a hypergraph in which each hyperedge has exactly $k$ vertices (or $k$-hyperedge). In particular, the complete $k$-uniform hypergraph on  $n$ vertices has all $k$-subsets of $\{1,2, \cdots, n \}$ as $k$-hyperedges.
A $k$-hypergraph access structure is represented by a $k$-uniform hypergraph $\mathcal{H}$  in which the participants correspond to the vertices of the hypergraph $\mathcal{H}$, and  a set of  vertices can reconstruct the secret value from their shares if they are connected by a $k$-hyperedge, and the set of non-adjacent vertices does not get any information on the secret. A $k$-hypergraph access structure is also called a {\it $k$-homogeneous access structure}. \\

A $k$-homogeneous access structure is determined by the family of minimal qualified subsets with exactly $k$ different participants or $k$-uniform hypergraphs $\mathcal{H}$. Recall that the qualified subsets in the $(k,n)$-threshold secret sharing scheme are formed by all the subsets with at least $k$ participants in the set of $n$ participants. For example, 
if we  consider the access structure on a set of five participants $P=\{p_1,p_2,p_3,p_4,p_5\}$ having minimal qualified subsets $A_1 =\{p_1,p_2,p_3\}, A_2 =\{p_2, p_3,p_4\}, A_3 =\{p_3,p_4,p_5\}$, we check that it is not $(3,n)$-threshold but $3$-homogeneous. Note that there is a one-to-one correspondence between $k$-uniform hypergraphs and $k$-homogeneous access structures. Also, there is a one-to-one correspondence between complete $k$-uniform hypergraphs and $(k,n)$-threshold access structures.\\

The secret sharing schemes for $k$-homogeneous access structures have been constructed by many authors based on various techniques.  In 1990, Benaloh and Leichter~\cite{BL} constructed a secret sharing scheme for the $k$-homogeneous access structure with total share size $O(n^k / \log n)$. 
 For the dense $k$-homogeneous access structure,  a much more efficient  secret sharing scheme was constructed by 
Beimel, Farras, and Mintz~\cite{BFM}  in 2012, in which the total share size is  $\tilde{O}(2^kk^kn^{2+\beta})$.
 Recently, in 2020, Beimel and Farras~\cite{BF1} proved that for almost all $k$-homogeneous access structures there exists a secret sharing scheme with maximum share size $2^{\tilde{O}(\sqrt{k\log n})}$, a linear secret  sharing scheme with normalized maximum share size $\tilde{O}(n^{(k-1)/2})$, and a multi-linear secret sharing scheme with normalized  maximum share size $\tilde{O}(\log^{k-1}n)$ for exponentially long secrets using Conditional Disclosure of Secrets (CDS) protocol.

\subsection{$k$-uniform access structures}

An access structure is defined as {\it forbidden $k$-homogeneous  (or forbidden $k$-hypergraph)} determined by a $k$-uniform hypergraph $\mathcal{H}$ in which  the set of vertices can reconstruct the secret value from their shares in the secret sharing scheme if they are connected by a $k$-hyperedge or their size is at least $k+1$. We study the complexity of realizing a forbidden $k$-homogeneous access structure by linear secret sharing schemes. A forbidden $k$-homogeneous access structure has been studied by many authors~\cite{AA, ABFNP, BKN,SS} under the terminology of $k$-uniform access structures. 
Secret sharing schemes for $k$-homogeneous structures  and  $k$-uniform structures are very similar. However, $k$-homogeneous access structures are harder to realize than $k$-uniform access structures. From the given secret sharing scheme realizing $k$-homogeneous access structures, we can construct secret sharing schemes realizing $k$-uniform access structures by giving a share of the  secret sharing schemes realizing $k$-homogeneous access structures and a share of  $(k+1,n)$-threshold secret sharing schemes. However, the total share size of the new scheme is slightly greater than the former scheme.\\

Recently, in 2018, Applebaum and Arkis~\cite{AA} constructed an efficient secret sharing scheme for  $k$-uniform access structures  using  multiparty Conditional Disclosure of Secrets (CDS) protocol. Later, in 2018, Beimel and Peter~\cite{BP} obtained that every $k$-uniform access structure with a binary secret could be realized by a secret sharing scheme in which the share size is $ \min  \{ (O(n/k))^{(k-1)/2}, O(n\cdot 2^{n/2})\} $. By improving their result, in 2019, Applebaum, Beimel, Farras, Nir, and Peter~\cite{ABFNP} obtained that every $k$-uniform access structure with a binary secret could be realized by a secret sharing scheme with share size $ 2^{\tilde{O}(\sqrt{k\log n})}$ by combining  CDS protocols and transformations. In 2020, Beimel, Farras, Mintz, and Peter~\cite{BFMP} obtained the lower bound on the max share size for sharing an one-bit secret in every linear secret sharing scheme realizing $k$-uniform access structures using CDS protocol.
In this paper, we provide efficient constructions 
 on the share size of linear secret sharing schemes for  sparse  $k$-uniform access structure  having at  most $n^{1+\beta}$ and dense 
   $k$-uniform access structure having at least $ {{n}\choose{k}}- n^{1+\beta}$, where $0 \leq \beta <1$,
  for a constant $k$ based on the hypergraph decomposition technique and the monotone span programs.

\section{Hypergraph Decomposition}
In this section, we describe the technique of hypergraph decomposition which plays  an important role for proving the main theorem. A hypergraph decompostion technique is a generalization of graph decomposition studied in~\cite{AA, BSSV, CG,DS} and first introduced in~\cite{BSSV}.
A hypergraph is defined as a pair $(V,E)$ where $V$ is a non-empty set of vertices and $E$ is a set of non-empty subsets of $V$ called hyperedges.
A  $k$-uniform hypergraph  is a hypergraph in which each hyperedge has exactly $k$ vertices.  In the hypergraph decomposition technique,  we first represent $k$-homogeneous access structure (or $k$-uniform access structure) as a $k$-uniform hypergraph $\mathcal{H}$. Then we decompose the $k$-uniform hypergraph in smaller sub-hypergraphs $H_1, H_2, \cdots, H_m$ for which we construct  efficient secret sharing schemes and such that all the hyperedges in $\mathcal{H}$ belong to at least one of  $H_i$, where $1 \leq i \leq m$. The secret sharing schemes for $k$-uniform hypergraphs are obtained as an union of the secret sharing schemes for all  sub-hypergraphs $H_1, H_2, \cdots, H_m$. The following  is the definition of the hypergraph decomposition in the graph theory.

\begin{definition}[Hypergraph Decomposition]
Let $\mathcal{H} =(V,E)$ be a hypergraph and let $\mathcal{H}_i (V_i, E_i)$ be a sub-hypergraph of a hypergraph $\mathcal{H}$ such that $V_i \subset V$ and $E_i \subset E$, where $1 \leq i \leq m$. The set  $\mathcal{F} = \{ \mathcal{H}_1, \mathcal{H}_2, \cdots, \mathcal{H}_m\}$ is said to be a decomposition of $\mathcal{H}$ if and only if each hyperedge in the hypergraph $\mathcal{H}$ belongs to at least one $\mathcal{H}_i$, where $1 \leq i \leq m$.
\end{definition}

In this paper, we consider sub-hypergraphs as the class of $k$-partite $k$-uniform hypergraphs for $k \geq 2$. 
 A $k$-uniform hypergraph is said to be $k$-partite if its vertex set can be partitioned into $k$ nonempty sets $V_1, V_2, \cdots, V_k$ in such a way that every hyperedge intersects every set $V_i$ of the partition in exactly one vertex, where $1\leq i \leq k$. In this paper, we utilize the following result that every $k$-uniform hypergraph can be decomposed into the set of sub-hypergraphs consisting of $O(\log n)$ $k$-partite $k$-uniform hypergraphs. It means that 
$k$-uniform hypergraph is covered by
  $k$-partite $k$-uniform hypergraphs of size $O(\log n)$.

\begin{lemma}\label{lemma:hyper}
Let $\mathcal{H} =(V,E)$ be a $k$-uniform hypergraph and let $\mathcal{H}_i (V_i, E_i)$ be a $k$-partite $k$-uniform sub-hypergraph of a hypergraph $\mathcal{H}$ such that $V_i \subset V$ and $E_i \subset E$, where $1 \leq i \leq m$. Then the $k$-uniform hypergraph $\mathcal{H}$ is covered by the set  $\mathcal{F} = \{ \mathcal{H}_1, \mathcal{H}_2, \cdots, \mathcal{H}_m\}$ consisting of $k$-partite $k$-uniform hypergraphs, where $m =O(\log n)$.
\end{lemma}

 \begin{proof}
 For every $i \in [O(\log n)]$,  take  random mapping $\phi_i : V \rightarrow \{1,2, \cdots, k\}$.  For  $i \in [O(\log n)]$,  let $X_i$ be the indicator random variable for the event the given hyperedge $e =(v_1,v_2, \cdots, v_k)$ in the hypergraph $\mathcal{H}$ appears as a hyperedge in the sub-hypergraph $\mathcal{H}_i$.
 For the random variable $X$ satisfying $X = \sum_{i=1}^{O(\log n)} X_i$, we have

 \begin{align*}
 \mathbb{E}(X)= \sum^{O(\log n)}_{i=1}Pr ( \{ \phi_i(v_1), \phi_i(v_2), \cdots, \phi_i(v_k) \} {\text \ covers \ } [k]) = \sum^{O(\log n)}_{i=1} \frac{k !}{k^k} = O(\log n)\cdot \frac{k !}{k^k} 
 \end{align*}
 
 Using the second variant of Chernoff bound, we obtain

 \begin{align*}
 &Pr\left(X \leq \left(1-(1-\frac{k^k}{k^{k}+1})\right) \cdot O(\log n) \cdot \frac{k !}{k^k} \right) = Pr\left(X \leq  O(\log n) \cdot \frac{k !}{k^k+1} \right)\\
 & \leq e^{\frac{-\left(1-\frac{k^k}{k^{k}+1}\right)^2 \cdot O(\log n) \cdot \frac{k !}{k^k}}{2}} \leq e^{-1.5 \ln (n^k)} = n^{-1.5k}. 
 \end{align*}
 
 Therefore, we derive that
 
 \begin{align*}
 n^k \cdot Pr\left(X \leq  O(\log n) \cdot \frac{k !}{k^k+1} \right) \leq n^k \cdot n^{-1.5k} = n^{-\frac{k}{2}} < 1
 \end{align*}
 
 We conclude that each hyperedge in the $k$-uniform hypergraph $\mathcal{H}$ belongs to at least one $k$-partite $k$-uniform sub-hypergraph  $\mathcal{H}_i$, where $1 \leq i \leq O(\log n)$. It means that the $k$-uniform hypergraph $\mathcal{H}$ is covered by the set  $\mathcal{F} = \{ \mathcal{H}_1, \mathcal{H}_2, \cdots, \mathcal{H}_m\}$ consisting of $k$-partite $k$-uniform hypergraphs, where $m =O(\log n)$.
 \end{proof}

 \section{Constructions for $k$-partite $k$-uniform hypergraphs}

 To study the complexity of realizing a  $k$-uniform access structure by linear secret sharing schemes, we utilize 
  the technique of hypergraph decomposition in which secret sharing schemes for $k$-uniform hypergraphs are obtained as a union of the secret sharing schemes for all  $k$-partite $k$-uniform sub-hypergraphs. 
  In 1993, Karchmer and Wigderson proved that if an access structure can be described by a  monotone span program, then it has an efficient linear secret sharing scheme.  In this section, we first construct the linear secret sharing schemes for $k$-partite $k$-uniform hypergraphs using monotone span programs.
  Using these constructions, we give more efficient linear secret sharing schemes for sparse and dense $k$-uniform hypergraphs in Section $6$.
In order to do so, we first need the following lemma about the construction of linear spaces corresponding to vertices.

\newpage

\begin{lemma}\label{r:lem1}Let $\mathcal{H}(V,E)$ be a $k$-partite $k$-uniform hypergraph, where $V$ is a set of vertices and $E$ is a set of hyperedges satisfying the following condition.
 Suppose that $V$ is partitioned into  $ A_1 \cup A_2 \cdots, \cup A_k$ with $|A_i|=m_i$.
Let $E$ be the family of subsets with exactly one vertex in common with each $A_i$ as follows.
 $$ E = \{ (a_{1,i_1}, \cdots, a_{k, i_k}) \ | \  a_{1, i_1} \in A_1,\  \cdots, \  a_{k,i_k}\in A_k \}.$$

 Suppose that every vertex in $A_k$ is contained in at most $d$ members in $E$ for some $ d\leq n$.
 Let $\mathbf{F}$ be a finite field with $| \mathbf{F} | \geq m_1+\cdots +m_{k-1}$. Let us denote $A_j = \{a_{j,1},  \cdots, a_{j,m_j} \}$, where $1\leq j \leq k$. For every $1 \leq j \leq k-1$, there exist a linear space $ V_{j,t} \subseteq \mathbf{F}^{(d+1)^{k-1}}$ corresponding to a vertex $  a_{j,t}$ in the set $A_j$, where $ 1 \leq t \leq m_j$. Also, there exists a vector $z_{k,i_k} \in  \mathbf{F}^{(d+1)^{k-1}} $ corresponding to each vertex $a_{k,i_k}$ in the set $A_k$
  such that
$$ z_{k,i_k} \in V_{1,i_1}, \cdots,  z_{k,i_k} \in V_{k-1,i_{k-1}} \  \Longleftrightarrow  \  \left(a_{1, i_1},  \cdots,  a_{k-1,i_{k-1}},  a_{k,i_k} \right) \in Q. $$

\noindent where $ 1 \leq i_1 \leq m_1,  \cdots$,  $1 \leq i_k \leq m_k$.
\end{lemma}

\begin{proof}
Let $v = (v_{(0,\cdots,0)},  \cdots, v_{(j_1,\cdots, j_{k-1})}, \cdots,  v_{(d,\cdots,d)})$ be a vector in  the finite field $\mathbf{F}^{(d+1)^{k-1}}$, where $ 0 \leq j_1,  \cdots, j_{k-1} \leq d$.
First, we construct the following polynomial of degree $(k-1)d$  whose coefficients correspond  to coordinates of the vector $v$. 
$$v(X_1,  \cdots, X_{k-1}) = \sum_{0\leq j_1,  \cdots j_{k-1} \leq d} \gamma_{j_1\cdots j_{k-1}}{X_1}^{j_1} \cdots {X_{k-1}}^{j_{k-1}}\ \in  \ \mathbf{F} [X_1,  \cdots,X_{k-1}]$$ 
in which the coefficient $\gamma_{j_1\cdots j_{k-1}}$  is same as the coordinate $v_{(j_1,\cdots, j_{k-1})}$ in a vector $v$ in $\mathbf{F}^{(d+1)^{k-1}}$, where $0 \leq j_1,    \cdots, j_{k-1} \leq d$.\\

\noindent  For  $1\leq j \leq k$, let us  consider an element $ \alpha_{j,{i_j}} $ in the  field   $\mathbf{F}$  corresponding to a vertex $a_{j,i_j} $ in  $A_j$, where $ 1 \leq i_j \leq m_j$.
Now we define a linear space $V_{j,i_j} \subseteq \mathbf{F}^{(d+1)^{k-1}}$ corresponding to a vertex $ a_{j,i_j}$ in  $A_j$, where $ 1\leq j \leq k-1$, as follows.
  For  every $1\leq j \leq k-1$, let us define the linear space $V_{j,i_j}$ as 
 the space of polynomials $P(X_1, \cdots, X_{k-1})$  of degree $(k-1)d$ satisfying   $P(X_1,  \cdots, X_{j-1}, \alpha_{j,i_{j}}, X_{j+1}, \cdots,X_{k-1} ) = 0$, where $ 1 \leq i_j \leq m_j$.\\
  
 \noindent Since every vertex in $A_k$ is contained in at most $d$ members in $Q$ for some $ d\leq n$, 
 for  each vertex $a_{k, i_{k}}$ in $A_k$, where $1 \leq i_{k} \leq m_k$, first we define the family of sets  in $Q$ containing  $a_{k, i_{k}}$
 as
 $$ Q_{i_k} = \{ (a_{1,i_{1,t}}, \cdots, a_{k-1,i_{k-1,t}}, a_{k, i_{k}}) \ | | \  a_{1, i_1,t} \in A_1, \cdots, \  a_{k-1,i_{k-1},t}\in A_{k-1} {\text \ and \ } \ 1 \leq t \leq d' \}$$
 
 \noindent  for some $d' \leq d$.\\
 
 \noindent Now let us  define a vector $z_{k,i_k} \in  \mathbf{F}^{(d+1)^{k-1}} $ corresponding to each vertex $a_{k, i_{k}}$ in the set $A_k$, where $1 \leq i_{k} \leq m_k$, whose coordinates   correspond to the coefficients of  the following polynomial of degree $(k-1)d$.
 
\begin{align*}
z_{k,i_k}(X_1,  \cdots, X_{k-1}) & =  (X_1 -\alpha_{1,i_{1,1}}) \cdots (X_1-\alpha_{1,i_{1,d'}}) \\
 & \ \ \  \  \ \ \ \  \  \ \ \ \  \  \ \ \ \  \  \ \vdots \\
 & \cdot (X_{k-1} -\alpha_{k-1,i_{k-1,1}}) \cdots (X_{k-1}-\alpha_{k-1,i_{k-1,d'}}).
\end{align*}

\noindent Then we obtain that  $z_{k,i_k} \in V_{1,i_1}, \cdots,  z_{k,i_k} \in V_{k-1,i_{k-1}}$  implies that 
 $$\alpha_{1,i_1} \in \{\alpha_{1, i_{1,1}}, \cdots,  \alpha_{1, i_{1, d'}} \},  \cdots, \alpha_{k-1,i_{k-1}} \in \{\alpha_{k-1, i_{k-1,1}},  \cdots, \alpha_{k-1, i_{k-1, d'}} \} .$$ 
 
 \noindent Then we conclude that $$\left(a_{1, i_1},  \cdots,  a_{k,i_k} \right) \in Q.$$
 
  \noindent This completes the proof of Lemma~\ref{r:lem1}.
\end{proof}
 
 First, we investigate  a linear secret sharing scheme for  $k$-partite $k$-uniform hypergraphs satisfying the following condition  based on the monotone span programs. We utilize the following lemma for constructing more efficient linear secret sharing schemes for spare $k$-partite $k$-uniform hypergraphs in Section $6$,
when the size of all $k$ parts is the same.

\begin{lemma}\label{r:lem2}Let $\mathcal{H}(V,E)$ be a $k$-partite $k$-uniform hypergraph, where $V$ is a set of vertices and $E$ is a set of hyperedges satisfying the following condition.
 Suppose that $V$ is partitioned into  $ A_1 \cup A_2 \cdots, \cup A_k$ with $|A_i|=m_i$.
Let $E$ be the family of subsets with exactly one vertex in common with each $A_i$ as follows.
 $$ E = \{ (a_{1,i_1}, \cdots, a_{k, i_k}) \ | \  a_{1, i_1} \in A_1,\  \cdots, \  a_{k,i_k}\in A_k \}.$$
 Suppose that every vertex  in $A_k$ is contained in at most $d$ members in $E$ for some $ d\leq n$. Then there exists a linear secret sharing scheme for  a $k$-uniform access structure $\Gamma$ determined by $\mathcal{H}(V,E)$ with total share size $m_k +(d+1)^{k-1}(m_1+\cdots+m_{k-1})$.
\end{lemma}

\begin{proof}

Let $\Gamma$ be a  $k$-uniform access structure   determined by $\mathcal{H}(V,E)$.  First, we construct a monotone span program accepting this  $k$-uniform access structure $\Gamma$ using $(d+1)^{k-1}$ rows labeled by $a_{j,i_j}$, where $1
\leq j \leq k-1$ and $1 \leq i_j\leq m_j$,  and one row labeled by $a_{k,i_k}$, where $1 \leq i_k\leq m_k$. Using Lemma~\ref{r:lem1}, there exists a linear space $V_{j,i_j} \subseteq \mathbf{F}^{(d+1)^{k-1}}$ corresponding to a vertex $ a_{j,i_j}$ in the set $A_j$, where $ 1 \leq j \leq k-1$.
Let us denote  the basis of the linear space $V_{j,i_j} \subseteq \mathbf{F}^{(d+1)^{k-1}}$ as $\{ v_{j, i_j,1},  \cdots, v_{j, i_j,(d+1)^{k-1}-1}\}$. \\

\noindent To construct $(d+1)^{k-1}$ rows labeled by $a_{j,i_j}$,  we consider the following  vector in  $ \mathbf{F}^{(d+1)^{k-1}+k}$ 
\begin{align}\label{basis1}
\{ v'_{j,i_j,1},  \cdots, v'_{j,i_j,(d+1)^{k-1}-1}, e_j'= (e_{k-j+1}, 0,0,\cdots, 0) \}.
\end{align}
\noindent where $v'_{j, i_j,l} = (0,\cdots,0,v_{j,i_j,l})$ is a vector in   $ \mathbf{F}^{(d+1)^{k-1} + k}$and $e_1,  \cdots, e_k$ are standard basis vectors in $\mathbf{F}^k$.\\

\noindent To construct one row labeled by $a_{k,i_k}$, we consider the following vector  in $\mathbf{F}^{(d+1)^{k-1}+k}$
\begin{align}\label{basis2}
z'_{k,i_k} = ( 1, 0, \cdots, 0, z_{k,i_k})
\end{align}

\noindent where the vector $z_{k,i_k} \in  \mathbf{F}^{(d+1)^{k-1}} $ is obtained  in Lemma~\ref{r:lem1}.\\

\noindent Let us set a target vector  in $\mathbf{F}^{(d+1)^{k-1}+k}$ as $$ (1, 1,\cdots, 1, {\bf{0_{(d+1)^{k-1}}}}) $$

\noindent   where ${\bf{0_{(d+1)^{k-1}}}}$ is a zero vector in $ \mathbf{F}^{(d+1)^{k-1}}$. \\

\noindent Using Lemma~\ref{r:lem1}, we obtain that 
$$ z_{k,i_k} \in V_{1,i_1},  \cdots,  z_{k,i_k} \in V_{k-1,i_{k-1}}.$$

\noindent Then the target vector 
 $   (1, 1, \cdots, 1, {\bf{0_{(d+1)^{k-1}}}}) $  must be  in the  span  of all vectors in  (\ref{basis1}) and (\ref{basis2}). Therefore we conclude that  a  $k$-uniform access structure $\Gamma$  can be accepted by this monotone span program, then  it has an efficient linear secret sharing scheme with total share size $m_k +(d+1)^{k-1}(m_1+\cdots+m_{k-1})$.
\end{proof}

 Also, we investigate  a linear secret sharing scheme for  $k$-partite $k$-uniform hypergraphs satisfying the following condition  based on the monotone span programs.  We utilize the following lemma for constructing more efficient linear secret sharing schemes for dense $k$-partite $k$-uniform hypergraphs in Section $6$,
when the size of all $k$ parts is the same.

\begin{lemma}\label{r:lem3} Let $\mathcal{G}(V,E)$ be a $k$-partite $k$-uniform hypergraph, where $V$ is a set of vertices and $E$ is a set of hyperedges satisfying the following condition.
 Suppose that $V$ is partitioned into  $ A_1 \cup \cdots, \cup A_k$ with $|A_1|= \cdots =|A_{k-1}|=n, |A_k|=m_k\leq n$.
Let $E$ be the family of subsets with exactly one vertex in common with each $A_i$ as follows.
 $$ E = \{ (a_{1,i_1},  \cdots, a_{k, i_k}) \ | \  a_{1, i_1} \in A_1, \cdots, \  a_{k,i_k}\in A_k \}.$$ 
 Suppose that every vertex in $A_k$ is contained in at least $n-d$ members in $E$ for some $ d\leq n$. Then there exists a linear secret sharing scheme for  a  $k$-uniform access structure $\Gamma'$ determined by $\mathcal{G}(V,E)$ with total share size $ 2m_k +(d+1)^{k-1}(k-1)n$.
\end{lemma}

\begin{proof} Let $\mathcal{U}$ be the family of all subsets with exactly one vertex in common with each part $A_i$, where $1\leq i \leq k$. First, let us consider
the complement of a family $E$, which is denoted by $\overline{E}$, consisting of all subsets in the given  universal family $\mathcal{U}$ that are not in $E$.  Since every vertex in $A_k$ is contained in at least $n-d$ members in $E$,  every vertex in $A_k$ must be contained in at most $d$ members in $\overline{E}$.\\

\noindent Using Lemma~\ref{r:lem1}, there exist a  linear space $V_{j,i_j} \subseteq \mathbf{F}^{(d+1)^{k-1}}$ corresponding to a vertex $ a_{j,i_j}$ in  $A_j$,  where $1\leq i_j \leq m_j$, $ 1\leq j \leq k-1$, and a vector $z_{k,i_k} \in  \mathbf{F}^{(d+1)^{k-1}} $ corresponding to a vertex $a_{k,i_k}$ in the set $A_k$
  such that
\begin{align}\label{lem3:basis}
z_{k,i_k} \in V_{1,i_1},  \cdots,  z_{k,i_k} \in V_{k-1,i_{k-1}} \  \Longleftrightarrow  \  \left(a_{1, i_1},  \cdots,  a_{k-1,i_{k-1}},  a_{k,i_k} \right) \in \overline{E}. 
\end{align}
\noindent where $ 1 \leq i_1, \cdots, i_{k-1} \leq n$, $1 \leq i_k \leq m_k$.\\

Let $\Gamma'$ be  a  $k$-uniform access structure  determined by $\mathcal{G}(V,E)$.
 Now we construct a monotone span program accepting this   $k$-uniform access structure $\Gamma'$ using $(d+1)^{k-1}$ rows labeled by $a_{j,i_j}$, where $1
\leq j \leq k-1$ and $1 \leq i_j\leq n$,  and two rows labeled by $a_{k,i_k}$, where $1 \leq i_k\leq m_k$.  Let us denote  the basis of the linear space $V_{j,i_j} \subseteq \mathbf{F}^{(d+1)^{k-1}}$ as $\{ v_{j, i_j,1},  \cdots, v_{j, i_j,(d+1)^{k-1}-1}\}$, where $1
\leq j \leq k-1$. \\

\noindent To construct $(d+1)^{k-1}$ rows labeled by $a_{j,i_j}$, where $1\leq j \leq k-1$ and $1 \leq i_j\leq n$, we consider the following  vectors in  $ \mathbf{F}^{(d+1)^{k-1}+k}$ 
\begin{align}\label{lem3:basis1}
\{ v'_{j,i_j,1},  \cdots, v'_{j,i_j,(d+1)^{k-1}-1}, e_j'= (e_{k-j+1}, 0, \cdots, 0) \}.
\end{align}
\noindent where $v'_{j, i_j,l} = (0,\cdots,0,v_{j,i_j,l})$ is a vector in   $ \mathbf{F}^{(d+1)^{k-1} + k}$and $e_1,  \cdots, e_k$ are standard basis vectors in $\mathbf{F}^k$.\\

\noindent To construct two rows labeled by $a_{k,i_k}$, we consider the following two vectors  in $\mathbf{F}^{(d+1)^{k-1}+k}$
\begin{align}\label{lem3:basis2}
z'_{k,i_k}= ( 0, 0, \cdots, 0, z_{k,i_k}), \ \ \ (1,0, \cdots, 0, {\bf{0_{(d+1)^{k-1}}}})
\end{align}
\noindent   where ${\bf{0_{(d+1)^{k-1}}}}$ is a zero vector in $ \mathbf{F}^{(d+1)^{k-1}}$.\\

\noindent Let us set a target vector  in $\mathbf{F}^{(d+1)^{k-1}+ k}$ as $$ (1,1, \cdots, 1, w) $$

\noindent   for some vector $w$ in $ \mathbf{F}^{(d+1)^{k-1}}$ which is not in 
all linear spaces $V_{j,i_j}$, where $1 \leq i_j \leq m_j$, $1\leq j\leq k-1$.\\

\noindent Using the equation~(\ref{lem3:basis}), $\left(a_{1, i_1},  \cdots,  a_{k-1,i_{k-1}},  a_{k,i_k} \right) \in E$ is equivalent to the  following statement
$$ z_{k,i_k}\not \in V_{1,i_1}, \cdots, z_{k,i_k} \not \in V_{k-1,i_{k-1}}.$$

\noindent Now  let us  consider a vector $w$   in the  span  of $\{ z_{k,i_k},  V_{1,i_1},   \cdots, V_{k-1,i_{k-1}} \}$, where $w \not \in V_{j,i_j}$  for every $1\leq j\leq k-1$ and $ 1 \leq i_j \leq n$. Then
the target vector 
 $   (1,1, \cdots, 1, w) $  must be  in the  span  of all vectors in (\ref{lem3:basis1}) and (\ref{lem3:basis2}). Therefore we conclude that a  $k$-uniform access structure $\Gamma'$, which is determined by $\mathcal{G}(V,E)$, can be accepted by this monotone span program, then  it has an efficient linear secret sharing scheme with total share size $2m_k +(d+1)^{k-1}(k-1)n$.
\end{proof}

\section{Proof of Theorem~\ref{main:mainthm1} }

In this section, we prove Theorem~\ref{main:mainthm1} by providing efficient constructions 
 on the share size of linear secret sharing schemes for sparse   $k$-uniform access structures   for a constant $k$. 
 To prove Theorem~\ref{main:mainthm1}, we utilize 
  the technique of hypergraph decomposition in which secret sharing schemes for $k$-uniform hypergraphs are obtained as a union of the secret sharing schemes for all  $k$-partite $k$-uniform sub-hypergraphs. 
We need the following lemma for constructing more efficient linear secret sharing schemes for spare $k$-partite $k$-uniform hypergraphs 
when the size of all $k$ parts is the same.\\

\begin{lemma}\label{main:lem1}Let $\mathcal{H}(V,E)$ be a $k$-partite $k$-uniform hypergraph, where $V$ is a set of vertices and $E$ is a set of hyperedges satisfying the following condition. Suppose that $V$ is partitioned into  $ A_1 \cup \cdots \cup A_k$ with $|A_i|=m_i$.
Let $E$ be the family of subsets with exactly one vertex in common with each $A_i$ as follows.
 $$ E = \{ (a_{1,i_1},  \cdots, a_{k, i_k}) \ | \  a_{1, i_1} \in A_1,\   \cdots, \  a_{k,i_k}\in A_k \}.$$
Suppose that $m_1= \cdots = m_{k-1} =n, \ m_k \leq n$, and
  every vertex in $A_k$ is contained in at most $d$ members in $E$ for some $ d\leq n$. 
If $d|A_k|^{k-1} \geq n^{k-1} \log^{k^2-2k+2} n$, then there exists a linear secret sharing scheme for  a $k$-uniform access structure $\Gamma$, which is determined by $\mathcal{H}(V,E)$,  with total share size $$O(\sqrt[k^2-2k+2]{n^{-k+1}|A_k|^{k^2-3k+3}d^{(k-1)^2}} \log^{k-1} n).$$ 
\end{lemma}

\begin{proof}
 Let  $d = n^{\tau}$ and $|A_k| = n^{\lambda} \leq n$, where $\tau = \log_n d$ and  $\lambda = \log_n |A_k|$.  From the condition 
$d|A_k|^{k-1} \geq n^{k-1} \log^{k^2-2k+2} n$, we obtain  
\begin{align}\label{condition1}
n^{\frac{\tau}{k^2-2k+2}+ \frac{(k-1)\lambda}{k^2-2k+2} - \frac{k-1}{k^2-2k+2}} \geq \log n.
\end{align}
 Let $\alpha = \frac{\lambda}{k^2-2k+2} - \frac{(k-1)\tau}{k^2-2k+2} + \frac{k^2-2k+1}{k^2-2k+2}$. In order to prove Lemma~\ref{main:lem1}, first we prove that there exists a partition of  $A_i$ into $l$ parts $A_{i,1},  \cdots, A_{i,l}$ of  size $n^{\alpha}$ for $1 \leq i \leq k-1$, where  $l = 2n^{1-\alpha}\ln n$, satisfying that
for every $1\leq i_1,  \cdots, i_k\leq l $, every vertex in $A_k$ is contained in at most  $2 n^{(k-1)\alpha +\tau-k+1}$ members in 
$$ E_{i_1,\cdots, i_{k-1}} = \{ (a_{1,i_1},  \cdots,a_{k-1, i_{k-1}} a_{k, i_k}) \ | \  a_{1, i_1} \in A_{1,i_1}, \cdots, \ a_{k-1,i_{k-1}}\in A_{k-1,i_{k-1}},  a_{k,i_k} \in A_{k} \}.$$

\noindent Now we choose $A_{1,i_1},  \cdots, A_{k-1, i_{k-1}}$   of size $n^{\alpha}$ independently with uniform distribution for every $1\leq i_1,  \cdots, i_k\leq l $. Then we have
\begin{align}\label{equation1}
Pr(A_i \neq \cup_{j=1}^l A_{i,j}) \leq \sum_{a\in A_i} Pr (a \not \in \cup_{j=1}^l A_{i,j}) & = \sum_{a \in A_i} \prod_{j=1}^l Pr(a \not \in A_{i,j}) \nonumber\\ & = \sum_{a\in A_i} \left(1 - \frac{n^{\alpha}}{n} \right)^l \leq \sum_{a\in A_i} e^{-\frac{l}{n^{1-\alpha}}} = n \cdot \frac{1}{n^2} = \frac{1}{n}
\end{align}
\noindent for every $1\leq i \leq k-1$.\\

\noindent Let  $x_{(i_1, \cdots, i_k)}$ be a value  in $X_{(i_1,\cdots, i_k)}$. 
For every vector $x=(x_{(j_1, \cdots,j_k)})_{(j_1, \cdots,j_k) \neq (i_1,\cdots,i_k)}$, let us consider 
$$p_x = Pr \left(X_{(i_1, \cdots, i_k)} = 1 \  | \ X_{(j_1, \cdots,j_k)}= x_{(j_1, \cdots,j_k)} {\text \ for \ all \ } (j_1, \cdots,j_k) \neq (i_1, \cdots,i_k ) \right).$$

\noindent From the equation  $n^{\frac{\tau}{k^2-2k+2}+ \frac{(k-1)\lambda}{k^2-2k+2} - \frac{k-1}{k^2-2k+2}} \geq \log n$, we obtain that

$$n^{\alpha} = n^{\frac{\lambda}{k^2-2k+2} - \frac{(k-1)\tau}{k^2-2k+2} + \frac{k^2-2k+1}{k^2-2k+2}} \leq \frac{n^{\lambda}}{\log^{k-1} n} < \frac{n}{k-1}.$$

\noindent It means that
$$n^{k-1}-n^{(k-1)\alpha} \leq \frac{k^2-2k}{k^2-2k+1} n^{k-1}.$$

\noindent Then we have 
$$p_x \leq \frac{n^{\tau}}{n^{k-1}-n^{(k-1)\alpha}} \leq \frac{k^2-2k+1}{k^2-2k}\frac{1}{n^{k-1-\tau}}.$$

\noindent 
Now  let us define the independent random variables as follows.
 
\begin{align*}
X'_{(i_1,\cdots, i_k)} = \left\{ \begin{array}{lll}
1 &   \mbox{if}   \ x_{(i_1,\cdots,i_k)}=1 \\ 1  {\text \ with \ probability \ } \frac{ \left(\frac{k^2-2k+1}{k^2-2k}\frac{1}{n^{k-1-\tau}} -p_x\right)}{(1-p_x)}   & \mbox{if}   \ x_{(i_1,\cdots,i_k)}= 0  \\ 0 & \mbox{otherwise} & \\
\end{array}\right.
\end{align*}

\noindent Then we obtain that 
\begin{align*}
 & Pr \left(X'_{(i_1, \cdots,i_k)} = 1 \  | \ X_{(j_1, \cdots, j_k)}= x_{(j_1, \cdots, j_k)} {\text \ for \ all \ } (j_1, \cdots, j_k) \neq (i_1, \cdots,i_k) \right) \\
 &= \frac{k^2-2k+1}{k^-2k}\frac{1}{n^{k-1-\tau}} .
 \end{align*}

\noindent  Then we have the expectation of the random variable  $X' = \sum_{i_1=1}^{n^{\alpha}} $$\cdots$ $\sum_{i_{k-1}=1}^{n^{\alpha}}$ $  X'_{(i_1, \cdots,i_{k-1})}$.
$$\mathbb{E}(X')= \frac{k^2-2k+1}{k^2-2k} n^{(k-1)\alpha}\frac{1}{n^{k-1-\tau}} = \frac {k^2-2k+1}{k^2-2k} n^{(k-1)\alpha + \tau-k+1}.$$

\noindent From the equation  $n^{\frac{\tau}{k^2-2k+2}+ \frac{(k-1)\lambda}{k^2-2k+2} - \frac{k-1}{k^2-2k+2}} \geq \log n,$ we obtain that \\
$$ n^{(k-1)\alpha+\tau-k+1}  = n^{\frac{\tau}{k^2-2k+2}+ \frac{(k-1)\lambda}{k^2-2k+2} - \frac{k-1}{k^2-2k+2}} \geq \log n.$$

\noindent By applying a chernoff bound to the random variable $X'$,  we conclude that
\begin{align}\label{equation2}
Pr(X > 2 n^{(k-1)\alpha +\tau-k+1}) \leq Pr(X' > 2 n^{(k-1)\alpha +\tau-k+1} ) \leq 2^{-(2 n^{(k-1)\alpha +\tau-k+1})} \leq 2^{-2 \log n} = \frac{1}{n^{2}}. 
\end{align}

\noindent Using the equations $(\ref{equation1})$ and $(\ref{equation2})$,  
there exist $A_{i,1},  \cdots, A_{i,l} \subset A_i$ of size $n^{\alpha}$ for $1 \leq i \leq k-1$, where  $l = 2n^{1-\alpha}\ln n$, such that the following holds:
 (1) $\bigcup_{j=1}^{l} A_{i,j} = A_i$ \   for \ $1 \leq i \leq k-1$
 (2) For every $1\leq i_1,  \cdots, i_k\leq l $, every vertex in $A_k$ is contained in at most  $2 n^{(k-1)\alpha +\tau-k+1}$ members in 
$ E_{i_1,\cdots, i_{k-1}} = \{ (a_{1,i_1}, \cdots, a_{k-1, i_{k-1}} a_{k, i_k}) \ | \  a_{1, i_1} \in A_{1,i_1}, \cdots, \ a_{k-1,i_{k-1}}\in A_{k-1,i_{k-1}},  a_{k,i_k} \in A_{k} \}.$ \\

Now we are ready to prove Theorem~\ref{main:lem1}. 
\noindent Apply  Lemma~\ref{r:lem2} with  $ E_{i_1,\cdots, i_{k-1}}$. Then 
 we conclude that
there exists a linear secret sharing scheme for  a $k$-uniform access structure $\Gamma$, which is determined by $\mathcal{H}(V,E)$,  with the following total share size 
\begin{align*}
&\sum_{i_1=1}^{l} \cdots \sum_{i_{k-1}=1}^{l}  \left(|A_k| +\left(2n^{(k-1)\alpha+\tau-k+1}+1\right)^{k-1}\left( |A_{1,i_1}|+ \cdots |A_{k-1,i_{k-1}}| \right) \right) \\
& = O \left(n^{(k-1)(1-\alpha)} \log^{k-1} n\left( n^{\lambda} + n^{(k-1)^2\alpha+\tau(k-1)-(k-1)^2+\alpha} \right)\right) \\
& = O\left(\sqrt[k^2-2k+2]{n^{-k+1}|A_k|^{k^2-3k+3}d^{(k-1)^2}} \log^{k-1} n \right).
\end{align*}
\noindent  This completes the proof of Lemma~\ref{main:lem1}.
\end{proof}

\noindent Using Lemma~\ref{main:lem1}, we give the following efficient linear secret sharing scheme for 
sparse  $k$-uniform access structure when the size of all $k$ parts is same.

\begin{lemma}\label{main:lem2}
Let $\mathcal{G}(V,E)$ be a $k$-partite $k$-uniform hypergraph, where $V$ is a set of vertices and $E$ is a set of hyperedges satisfying the following condition.
Suppose that $V$ is partitioned into  $ A_1 \cup \cdots \cup A_k$ with $|A_i|=m_i$.
Let $E$ be the family of subsets with exactly one vertex in common with each $A_i$ as follows.
 $$ E = \{ (a_{1,i_1},  \cdots, a_{k, i_k}) \ | \  a_{1, i_1} \in A_1,\   \cdots, \  a_{k,i_k}\in A_k \}.$$
 Suppose that $m_1= \cdots = m_{k} =n$ and there are at most $n^{1+\beta}$ subsets for some $0 \leq \beta < 1$ in $E$.
Then there exists a linear secret sharing scheme for a sparse $k$-uniform access structure $\Gamma$, which is determined by $\mathcal{G}(V,E)$,  with total share size $$O(n^{\frac{k^2-3k+2}{k^2-2k+2}+ \left(\frac{k^2-3k+3}{k^2-2k+2}\right)\beta} \log^{k} n ).$$ 
\end{lemma}

\begin{proof}

Let us consider a partition  of the participants in $A_k$ into $ \log n$ sets according to the number of sets in $E$ containing each participant. Let us define the $s$-th set ${A_{k}}^{(s)}$ as
$${A_{k}}^{(s)} = \{  \ v \in \mathcal{P} \ | \ \frac{n}{2^{s+1}} \leq {\text \ number\ of \ members \ in \ E \ containing  \ a \ participant \ v \ }  \leq \frac{n}{2^s} \}$$
\noindent for $s=0,1, \cdots, \log n -1$. 
Since there are at most $n^{1+\beta}$ subsets for some $0 \leq \beta < 1$ in $Q$ and
 the  number of  members  in  $E$  containing  every participant
  in the $s$-th set ${A_{k}}^{(s)}$ is at least $\frac{n}{2^{s+1}}$,
we derive that the number of participants in the $s$-th set ${A_{k}}^{(s)}$ is at most $\frac{n^{1+\beta}}{\frac{n}{2^{s+1}}} = 2^{s+1}n^{\beta}$
for $s=0,1, \cdots, \log n -1$.\\

\noindent If we apply Lemma~\ref{main:lem1} with
$$ Q_s = \{ (a_{1,i_1}, \cdots, a_{k-1, i_{k-1}}, a_{k, i_k}) \ | \  a_{1, i_1} \in A_1,\  \cdots, \  a_{k-1,i_{k-1}} \in A_{k-1},  a_{k,i_k}\in {A_{k}}^{(s)} \},$$
 where $|A_1|=  \cdots = |A_{k-1}| = n $ and $|{A_{k}}^{(s)}| \leq 2^{s+1}n^{\beta}$, then we conclude that
there exists a linear secret sharing scheme for  a  $k$-uniform access structure $\Gamma$, which is determined by $\mathcal{G}(V,E)$, with the following total share size 

\begin{align*}
&  O\left(\sqrt[k^2-2k+2]{n^{-k+1}|{A_{k}}^{(s)}|^{k^2-3k+3}d^{(k-1)^2}} \log^{k-1} n \right) \times \log n\\
 =\  & O\left(\sqrt[k^2-2k+2]{n^{-k+1}|2^{s+1}n^{\beta}|^{k^2-3k+3}\left(\frac{n}{2^s}\right)^{(k-1)^2}} \log^{(k-1)} n \right) \times \log n\\
 =\  & O\left(\sqrt[k^2-2k+2]{n^{k^2-3k+2} n^{{\beta}(k^2-3k+3)}}
 \log^{k} n \right).
\end{align*}

\noindent This completes the proof of Lemma~\ref{main:lem2}.
\end{proof}

\noindent To prove Theorem~\ref{main:mainthm1}, 
now we utilize 
  the technique of hypergraph decomposition described in Section $4$. Using Lemma~\ref{lemma:hyper}, 
  we obtain that every $k$-uniform hypergraph can be decomposed into the set of sub-hypergraphs consisting of $O(\log n)$ $k$-partite $k$-uniform hypergraphs. It means that 
$k$-uniform hypergraph is covered by
  $k$-partite $k$-uniform hypergraphs of size $O(\log n)$.
Let us consider the collection of the sets of participants into $k$ parts $A_1^t, \cdots, A_k^t$ for every $ 1 \leq t \leq O(\log n)$. For every $ 1 \leq t \leq O(\log n)$,
let us define the family $E_t$ of subsets with exactly one vertex in common with each $A_i^t$ as 
 $$ E_t = \{ (a_{1,i_1},  \cdots, a_{k, i_k}) \ | \  a_{1, i_1} \in A_1^t,\   \cdots, \  a_{k,i_k}\in A_k^t \},$$
 
 \noindent where $|A_1^t|=\cdots =|A_k^t|$.\\

Applying  Lemma~\ref{main:lem2} with  $ E_t$,
 we conclude that
there exists a linear secret sharing scheme for  a  $k$-uniform access structure $\Gamma$, which is determined by $E = \bigcup_{t=1}^{O(\log n)} E_t $,  with the following total share size

\begin{align*}
 O\left(\sqrt[k^2-2k+2]{n^{k^2-3k+2} n^{{\beta}(k^2-3k+3)}}
 \log^{k+1} n \right).
\end{align*}

\noindent This completes the proof of Theorem~\ref{main:mainthm1}.

\section{Proof of Theorem~\ref{main:mainthm2} }

In this section, we prove Theorem~\ref{main:mainthm2} by providing efficient constructions 
 on the share size of linear secret sharing schemes for dense   $k$-uniform access structures for a constant $k$. 
 To prove Theorem~\ref{main:mainthm2}, we utilize 
  the technique of hypergraph decomposition in which secret sharing schemes for $k$-uniform hypergraphs are obtained as a union of the secret sharing schemes for all  $k$-partite $k$-uniform sub-hypergraphs. 
We need the following lemma for constructing more efficient linear secret sharing schemes for dense $k$-partite $k$-uniform hypergraphs 
when the size of all $k$ parts is the same.\\

\begin{lemma}\label{main2:lem2}
Let $\mathcal{H}(V,E)$ be a $k$-partite $k$-uniform hypergraph, where $V$ is a set of vertices and $E$ is a set of hyperedges satisfying the following condition.
Suppose that $V$ is partitioned into  $ A_1 \cup \cdots \cup A_k$ with $|A_i|=m_i$.
Let $E$ be the family of subsets with exactly one vertex in common with each $A_i$ as follows.
 $$ E = \{ (a_{1,i_1},  \cdots, a_{k, i_k}) \ | \  a_{1, i_1} \in A_1,\   \cdots, \  a_{k,i_k}\in A_k \}.$$
 Suppose that $m_1= \cdots = m_{k-1} =n, \ m_k \leq n$, and
  every vertex in $A_k$ is contained in at least $n-d$ members in $E$ for some $ d\leq n$. 
If $d|A_k|^{k-1} \geq n^{k-1} \log^{k^2-2k+2} n$, then there exists a linear secret sharing scheme for  a  $k$-uniform access structure $\Gamma$, which is determined by $\mathcal{H}(V, E)$,  with total share size $$O(\sqrt[k^2-2k+2]{n^{-k+1}|A_k|^{k^2-3k+3}d^{(k-1)^2}} \log^{k-1} n).$$ 
\end{lemma}

\begin{proof}
Let $\mathcal{U}$ be the family of all subsets with exactly one vertex in common with each part $A_i$, where $1\leq i \leq k$. First, let us consider
the complement of a family $E$, which is denoted by $\overline{E}$, consisting of all subsets in the given  universal family $\mathcal{U}$ that are not in $E$.  Since every participant in $A_k$ is contained in at least $n-d$ members in $E$,  every participant in $A_k$ must be contained in at most $d$ members in the complement $\overline{E}$. Now we apply Lemma~\ref{main:lem1}.\\

\noindent Let  $d = n^{\tau}$ and $|A_k| = n^{\lambda} \leq n$, where $\tau = \log_n d$ and  $\lambda = \log_n |A_k|$.  From the condition 
$d|A_k|^{k-1} \geq n^{k-1} \log^{k^2-2k+2} n$, we obtain  
\begin{align}\label{condition1:lem2}
n^{\frac{\tau}{k^2-2k+2}+ \frac{(k-1)\lambda}{k^2-2k+2} - \frac{k-1}{k^2-2k+2}} \geq \log n.
\end{align}
 Let $\alpha = \frac{\lambda}{k^2-2k+2} - \frac{(k-1)\tau}{k^2-2k+2} + \frac{k^2-2k+1}{k^2-2k+2}$. In the same way of the proof of Lemma~\ref{main:lem1}, there exist $A_{i,1},  \cdots, A_{i,l} \subset A_i$ of size $n^{\alpha}$ for $1 \leq i \leq k-1$, where  $l = 2n^{1-\alpha}\ln n$, such that the following holds:
 (1) $\bigcup_{j=1}^{l} A_{i,j} = A_i$ \   for \ $1 \leq i \leq k-1$
 (2) For every $1\leq i_1, \cdots, i_k\leq l $, every participant in $A_k$ is contained in at most  $2 n^{(k-1)\alpha +\tau-k+1}$ members in  $\overline{E}_{i_1,\cdots, i_{k-1}}\subseteq \overline{E}$, where
$ \overline{E}_{i_1,\cdots, i_{k-1}} = \{ (a_{1,i_1},  \cdots,a_{k-1, i_{k-1}}, a_{k, i_k}) \ | \  a_{1, i_1} \in A_{1,i_1}, \cdots, \ a_{k-1,i_{k-1}}\in A_{k-1,i_{k-1}},  a_{k,i_k} \in A_{k} \}.$ \\

\noindent Let us consider
the complement of the family $ \overline{E}_{i_1, \cdots, i_{k-1}}$, which is denoted by $ E_{i_1,\cdots, i_{k-1}}$, consisting of all subsets in the  universal family that are not in $ \overline{E}_{i_1,\cdots, i_{k-1}}$. Since every participant in $A_k$ is contained in at most  $2 n^{(k-1)\alpha +\tau-k+1}$ members in  $\overline{E}_{i_1,\cdots, i_{k-1}}\subseteq \overline{E}$, every participant in $A_k$ is contained in at least $n^{\alpha}-2 n^{(k-1)\alpha +\tau-k+1}$ members in  $E_{i_1, \cdots, i_{k-1}}\subseteq E.$\\

\noindent Apply  Lemma~\ref{r:lem3} with  $ E_{i_1,\cdots, i_{k-1}}$. Then 
 we conclude that
there exists a linear secret sharing scheme for  a  $k$-uniform access structure $\Gamma$, which is determined by $E$,  with the following total share size 

\begin{align*}
&\sum_{i_1=1}^{l}\sum_{i_2=1}^{l} \cdots \sum_{i_{k-1}=1}^{l}  \left(2|A_k|\ +  \ n^{\alpha}(k-1)\left(2n^{(k-1)\alpha+\tau-k+1}+1\right)^{k-1} \right) \\
& = O \left(n^{(k-1)(1-\alpha)} \log^{k-1} n\left( n^{\lambda} + n^{(k-1)^2\alpha+\tau(k-1)-(k-1)^2+\alpha} \right)\right) \\
& = O\left(\sqrt[k^2-2k+2]{n^{-k+1}|A_k|^{k^2-3k+3}d^{(k-1)^2}} \log^{k-1} n \right).
\end{align*}
\noindent  This completes the proof of Lemma~\ref{main2:lem2}.
\end{proof}

\noindent Using Lemma~\ref{main2:lem2},  we give the following efficient linear secret sharing scheme for 
dense  $k$-uniform access structure when the size of all $k$ parts is same.

\newpage

\begin{lemma}\label{main2:lem3}Let $\mathcal{G}(V,E)$ be a $k$-partite $k$-uniform hypergraph, where $V$ is a set of vertices and $E$ is a set of hyperedges satisfying the following condition.
Suppose that $V$ is partitioned into  $ A_1 \cup \cdots \cup A_k$ with $|A_i|=m_i$.
Let $E$ be the family of subsets with exactly one vertex in common with each $A_i$ as follows.
 $$ E = \{ (a_{1,i_1},  \cdots, a_{k, i_k}) \ | \  a_{1, i_1} \in A_1,\   \cdots, \  a_{k,i_k}\in A_k \}.$$
 Suppose that $m_1= \cdots = m_{k} =n$ and there are at least ${{n}\choose{k}} - n^{1+\beta}$ subsets for some $0 \leq \beta < 1$ in $E$.
Then there exists a linear secret sharing scheme for  a  $k$-uniform access structure $\Gamma$, which is determined by $\mathcal{G}(V,E)$,  with total share size $$O(n^{\frac{k^2-3k+2}{k^2-2k+2}+ \left(\frac{k^2-3k+3}{k^2-2k+2}\right)\beta} \log^{k} n ).$$ 
\end{lemma}

\begin{proof}
Let $\mathcal{U}$ be the family of all subsets with exactly one vertex in common with each part $A_i$, where $1\leq i \leq k$. First, let us consider
the complement of a family $E$, which is denoted by $\overline{E}$, consisting of all subsets in the given  universal family $\mathcal{U}$ that are not in $E$.  Since there are  at least ${{n}\choose{k}} - n^{1+\beta}$  subsets for some $0 \leq \beta < 1$ in $E$,  there are at most $n^{1+\beta}$ subsets for some $0 \leq \beta < 1$ in the complement $\overline{E}$. Now we apply Lemma~\ref{main:lem2}.\\

\noindent In the same way of the proof of Lemma~\ref{main:lem2}, let us consider a partition  of the vertices in $A_k$ into $ \log n$ sets according to the number of sets in $\overline{E}$ containing each vertex. Let us define the $s$-th set ${A_{k}}^{(s)}$ as
$${A_{k}}^{(s)} = \{  \ v \in \mathcal{P} \ | \ \frac{n}{2^{s+1}} \leq {\text \ number\ of \ members \ in \ \overline{E} \ containing  \ a \ vertex \ v \ }  \leq \frac{n}{2^s} \}$$
\noindent for $s=0,1, \cdots, \log n -1$. 
Since there are at most $n^{1+\beta}$ subsets for some $0 \leq \beta < 1$ in $\overline{E}$ and
 the  number of  members  in  $\overline{E}$  containing  every vertex
  in the $s$-th set ${A_{k}}^{(s)}$ is at least $\frac{n}{2^{s+1}}$,
we derive that the number of vertices in the $s$-th set ${A_{k}}^{(s)}$ is at most $\frac{n^{1+\beta}}{\frac{n}{2^{s+1}}} = 2^{s+1}n^{\beta}$
for $s=0,1, \cdots, \log n -1$.\\

\noindent For $s=0,1, \cdots, \log n -1$, let us define the family $\overline{E}_s \subseteq \overline{E}$ of subsets with exactly one vertex in common with $A_1,  \cdots, A_{k-1}, {A_k}^{(s)}$ as
\noindent $$ \overline{Q}_s = \{ (a_{1,i_1}, \cdots, a_{r-1, i_{r-1}}, a_{r, i_r})  \ | \  a_{1, i_1} \in A_1,\   \cdots, \  a_{k-1,i_{k-1}} \in A_{k-1},  a_{k,i_k}\in {A_{k}}^{(s)} \} \  \subseteq \overline{E} ,$$
 where $|A_1|=  \cdots = |A_{k-1}| = n $ and $|{A_{k}}^{(s)}| \leq 2^{s+1}n^{\beta}$.
  Let us consider
the complement of the family $ \overline{E}_{s}$, which is denoted by $ E_{s}$, consisting of all subsets in the  universal family that are not in $ \overline{E}_{s}$. \\

\noindent Since every participant in $A_k$ is contained in at most $\frac{n}{2^s}$ in $\overline{E}_s$, every participant in $A_k$ is contained in at least $n -\frac{n}{2^s}$ in $E_s$.
 If we apply Lemma~\ref{main2:lem2} with $E_s$, then we conclude that
there exists a linear secret sharing scheme for  a  $k$-uniform access structure $\Gamma$, which is determined by $E$, with the following total share size 

\begin{align*}
&  O\left(\sqrt[k^2-2k+2]{n^{-k+1}|{A_{k}}^{(s)}|^{k^2-3k+3}d^{(k-1)^2}} \log^{k-1} n \right) \times \log n\\
 =\  & O\left(\sqrt[k^2-2k+2]{n^{-k+1}|2^{s+1}n^{\beta}|^{k^2-3k+3}\left(\frac{n}{2^s}\right)^{(k-1)^2}} \log^{(k-1)} n \right) \times \log n\\
 =\  & O\left(\sqrt[k^2-2k+2]{n^{k^2-3k+2} n^{{\beta}(k^2-3k+3)}}
 \log^{k} n \right).
\end{align*}

\noindent This completes the proof of Lemma~\ref{main2:lem3}.
\end{proof}

\noindent To prove Theorem~\ref{main:mainthm2}, now we utilize 
  the technique of hypergraph decomposition described in Section $4$. Using Lemma~\ref{lemma:hyper},   
$k$-uniform hypergraph is covered by
  $k$-partite $k$-uniform hypergraphs of size $O(\log n)$.
Let us consider the collection of the sets of participants into $k$ parts $A_1^t, \cdots, A_k^t$ for every $ 1 \leq t \leq O(\log n)$. For every $ 1 \leq t \leq O(\log n)$,
let us define the family $E_t \subseteq E$ of subsets with exactly one vertex in common with each $A_i^t$ as 
 $$ E_t = \{ (a_{1,i_1},  \cdots, a_{k, i_k}) \ | \  a_{1, i_1} \in A_1^t,\   \cdots, \  a_{k,i_k}\in A_k^t \},$$
 
 \noindent where $|A_1^t|= \cdots =|A_k^t|$.\\

Applying  Lemma~\ref{main2:lem3} with  $ E_t$,
 we conclude that
there exists a linear secret sharing scheme for  a  $k$-uniform access structure $\Gamma$, which is determined by $E = \bigcup_{t=1}^{O(\log n)} E_t $,  with the following total share size

\begin{align*}
 O\left(\sqrt[k^2-2k+2]{n^{k^2-3k+2} n^{{\beta}(k^2-3k+3)}}
 \log^{k+1} n \right).
\end{align*}

\noindent This completes the proof of Theorem~\ref{main:mainthm2}.

\section{Conclusion}

In this paper,  we investigated efficient constructions 
 on the total share size of linear secret sharing schemes for sparse and dense $k$-uniform access structures (or forbidden $k$-homogeneous access structures)  for a constant $k$ using the hypergraph decomposition technique and  the monotone span programs.\\
 
 An access structure is ideal if there exists an ideal secret sharing scheme  realizing it. The characterization of the ideal access structures is one of the important problems in the secret sharing scheme. The characterization problems of ideal  access structures  have been studied by many authors~\cite{CT,FP1,FP2,JZB,DS1,TD}. An open problem is the search for new techniques to characterize the ideal $k$-homogeneous access structures.

\end{document}